\documentclass[12pt, a4paper]{amsart}
\usepackage{amsmath, amssymb, amsthm}
\usepackage{comment, empheq, xcomment}
\usepackage{mathdots, breqn, xcolor}
\usepackage[colorlinks]{hyperref}
\usepackage{mathtools}
\usepackage[capitalize,nameinlink,noabbrev]{cleveref}
\usepackage{dsfont}
\numberwithin{equation}{section}
\newtheorem{theorem}{Theorem}[section]
\newtheorem{lemma}[theorem]{Lemma}
\newtheorem{prop}[theorem]{Proposition}
\newtheorem{conj}[theorem]{Conjecture}

\theoremstyle{definition}

\newtheorem{remark}[theorem]{Remark}

\crefname{equation}{equation}{equations}
\crefname{prop}{Proposition}{Propositions}

\newcommand{\R}{\mathbb{R}}

\newcommand{\C}{\mathbb{C}}
\newcommand{\Q}{\mathbb{Q}}
\newcommand{\Z}{\mathbb{Z}}

\newcommand{\SL}{\mathrm{SL}}

\newcommand{\tr}{\mathrm{tr}}

\newcommand{\dist}{\mathrm{dist}}

\newcommand{\n}[1]{\left\Vert #1\right\Vert }

\newcommand{\spann}{\mathrm{span}}

\newcommand{\wt}{\widetilde}

\usepackage{graphicx} 
\usepackage[color,notref,notcite,final]{showkeys}

\title{Lifting all elements in $\mathrm{SL}_n(\mathbb{Z}/q\mathbb{Z})$}
\subjclass[2020]{Primary: 11F06; Secondary: 11J25, 20H05}
\keywords{Arithmetic groups, optimal lifting, strong approximation, Bohr sets, power residues}
\author{Amitay Kamber and P\'eter P. Varj\'u}
\thanks{The authors have received funding from the European Research Council (ERC) under the European Union’s Horizon 2020 research and innovation programme (grant agreement No. 803711).
PV was supported by The Royal Society.}
\address{Centre for Mathematical Sciences, Wilberforce Road, Cambridge CB3 0WB, UK.}
\email{amitayka@gmail.com}
\email{pv270@dpmms.cam.ac.uk}

\begin{document}

\begin{abstract}
We show that every element of $\mathrm{SL}_{n}(\mathbb{Z}/q\mathbb{Z})$ can be lifted to an element of $\mathrm{SL}_{n}(\mathbb{Z})$ of norm at most $Cq^2\log q$, while there exists an element such that every lift of it is of norm at least $q^{2+o(1)}$. This should be compared to the recent result that almost every element has a lift of norm bounded by $q^{1+1/n+o(1)}$.

The main step in the proof is showing that for every $q$, there is a small element in $(\mathbb{Z}/q\mathbb{Z})^\times$ with a large $n$-th root, which is a result of independent interest.
We use tools from additive combinatorics including Bohr sets
in the proof.
\end{abstract}
\maketitle

\section{Introduction}

Let $n\ge2$ be a natural number, and let $q\to\infty$ among integers. We treat $n$ as fixed throughout this work, and all the constants may depend on it.

By strong approximation, the modulo $q$ map, $\pi_{q}\colon\SL_{n}(\Z)\to\SL_{n}(\Z/q\Z)$ is onto. We study this map as follows. Let $\n{\cdot}\colon\SL_{n}(\Z)\to\R_{\ge0}$ be the operator norm. Given $T>0$, the set 
\[
F_T:=\left\{ \gamma\in\SL_{n}(\Z):\n{\gamma}\le T\right\}
\]
is finite. Therefore, restricting the map $\pi_q$ to $F_T$ gives us a finite problem of throwing balls into bins.
In this work, we consider the problem of when the balls fill all of the bins, but first, let us remark on the related problem of filling \emph{almost all} of the bins.

By the pigeon-hole principle, filling almost all the bins can only happen if $|F_T| \ge (1+o(1))|\SL_n(\Z/q\Z)|$. To make this explicit, we notice that 
\[
|\SL_n(\Z/q\Z)|=q^{n^2-1+o(1)},
\]
and Duke-Rudnick-Sarnak (\cite[Example 1.6]{duke1993density}) proved that for some constant $C_n$,
\begin{equation}\label{eq:DRS}
|F_T| =C_{n}T^{n^{2}-n}(1+o(1)).    
\end{equation}

This implies that to fill almost all the bins, or alternatively, to lift almost all the elements of $\SL_n(\Z/q\Z)$ to $\SL_n(\Z)$, one needs $T\ge q^{1+1/n+o(1)}$. 

Recently, it was proven that this order of magnitude for $T$ is also sufficient,
at least when $q$ is square-free.
\begin{theorem}[\cite{jana2022eisenstein-average} following \cite{assing2022density}]
\label{thm:average case} For every $\epsilon>0$, the following holds if $q$ is square-free and large enough in terms of $\epsilon$ and $n$.
For every $x\in\SL_{n}(\Z/q\Z)$
outside a set of at most $\epsilon\left|\SL_{n}(\Z/q\Z)\right|$
exceptions, there is $\gamma\in\SL_{n}(\Z)$ with $\pi_{q}(\gamma)=x$,
and $\|\gamma\|\le q^{1+1/n+\epsilon}$.
\end{theorem}

In this work, we consider the problem of lifting \emph{every} element and essentially solve it, by the following two theorems.
\begin{theorem}
\label{thm:lower bound}For every $\epsilon>0$, there is a constant
$C_{n,\epsilon}$ such that for every $q$, there exists an element $x\in\SL_{n}(\Z/q\Z)$
with 
\[
\min\{\n{\gamma}:\gamma\in\SL_{n}(\Z),\pi_{q}(\gamma)=x\}\ge C_{n,\epsilon}q^{2-\epsilon}.
\]
\end{theorem}

\begin{theorem}
\label{thm:upper bound}There is a constant $C_{n}$ such that for
every $q$ and every $x\in\SL_{n}(\Z/q\Z)$, there is $\gamma\in\SL_{n}(\Z)$
with $\pi_{q}(\gamma)=x$ and $\n{\gamma}\le C_{n}q^{2}\log q$. 
\end{theorem}

\cref{thm:lower bound}  shows that the worst case is quite far from the average case. Indeed, the ratio between the exponents for the almost-diameter and the diameter is $2n/(n+1)$, which is $4/3$ for $n=2$, and tends to $2$ as $n$ grows. One should notice that by a simple argument, the ratio is bounded by $2$, i.e., every element has a lift of norm bounded by $q^{2+2/n+o(1)}$.\footnote{The argument is not entirely trivial since $\|\gamma\|\ne \|\gamma^{-1}\|$, and goes as follows. For $0<\epsilon<1/2$ and $q$ large enough, the set $A=\{\pi_q(\gamma):\|\gamma\|\le q^{1+1/n+\epsilon}\}$ satisfies $A>\frac{1}{2}|\SL_n(\Z/q\Z)|$. Therefore, for every $x\in \SL_n(\Z/q\Z)$, $A\cap (A^{-1}x)\ne \phi$. So there are $\gamma_1,\gamma_2$ with $\|\gamma_i\|\le q^{1+1/n+\epsilon}$, $\pi_q(\gamma_1)= \pi_q(\gamma_2^{-1})x$. Then $\gamma_2\gamma_1$ is the required lift of $x$.}
Therefore, the worst case is close to being as bad as possible, a phenomenon Sarnak calls \emph{big holes}. Interestingly, the same factor $2n/(n+1)$ appears (implicitly) in the work of Harman \cite{harman1990approximation}, see the discussion below.

We provide some more context to \cref{thm:lower bound,thm:upper bound}. Studying upper bounds for lifts is closely related to expander graphs. Indeed, for $n\ge3$, let $S\subset\SL_{n}(\Z)$ be a generating set. Then Margulis, following Kazhdan, showed that the Cayley graphs with underlying group $\SL_{n}(\Z/q\Z)$ and generators $\pi_{q}(S)$ are a family of expanders \cite{margulis1973explicit}, \cite[\S 3.3]{lubotzky1994discrete}. This implies that there exists a constant $\alpha_{n}$ such that every $x\in\SL_{n}(\Z/q\Z)$ can be lifted to an element of norm $\le q^{\alpha_{n}}$. A similar bound follows from Selberg's work for $n=2$ \cite[\S 4.4]{lubotzky1994discrete}. One can make the bound somewhat more explicit, using the work of Gorodnik and Nevo \cite[Theorem~1.1]{gorodnik2012lifting} which uses the mean ergodic theorem for actions of Lie groups. Their very general work specialized to this case implies that one can find a lift of norm $\le C_{n}q^{2(n^2-1)/n}$ for $n\ge3$. 

It seems that the particular question above first appeared in a letter of Sarnak \cite{sarnak2015lettermiller}, who proved \cref{thm:upper bound} for the $n=2$ case, with the slightly weaker bound $Cq^{2}\log^{3}q$ \cite[equation (20)]{sarnak2015lettermiller}.
(Sarnak's argument yields, in fact, $Cq^{2}\log^{2}q$ if one uses the best
known bound for the Jacobsthal function due to Iwaniec \cite{iwaniec1978}.)
Finally, before this work, the best general bound was $q^{2+2/n+o(1)}$ based on \cref{thm:average case} as explained above.

Regarding \cref{thm:lower bound}, much less was known. The best lower bound before this work is based on the pigeon-hole argument stated above, which implies that there is an element $x\in \SL_n(\Z/q\Z)$ such that every lift of it is of norm $\ge q^{1+1/n+o(1)}$. For $n=2$, and in the special case that $q$ is a power of $2$,
Sarnak proved \cref{thm:lower bound}, and stated his belief that the same is true for every $q$ \cite[Page 11]{sarnak2015lettermiller}. \cref{thm:lower bound} confirms this and, moreover, extends it to $n\ge3$.

The question we study can be greatly generalized as follows. Assume that $G$ is a group that has a gauge function $l\colon G\to\R_{\ge0}$,
and $X_{q}$ are a sequence of finite sets on which $G$ acts. Then one can define a ``distance-like function'' on $X_{q}$ by \footnote{The function will be a metric if the conditions $l(\gamma_{1}\gamma_{2})\le l(\gamma_{1})+l(\gamma_{2})$,
$l(e)=0$, $l(\gamma)=l(\gamma^{-1})$, $l(g)> 0$ for $\gamma\ne e$ hold.}
\[
d(x,y):=\min\{l(\gamma):\gamma\cdot x=y\}.
\]

One may then ask for the properties of this space, and in particular about its ``diameter'', ``almost-diameter'', etc. 

In our case, $G=\SL_{n}(\Z)$, $l(\gamma)=\log(\|\gamma\|)$ and $X_{q}=\SL_{n}(\Z/q\Z)$\footnote{Notice that $l(\gamma)\ne l(\gamma^{-1})$ if $n\ge 3$, so this is not a proper metric space.}.
This case has the added property that the space $X_q$ is symmetric (that is, $G$ has a natural action on it on both sides). 

This problem is specifically interesting for general arithmetic groups, acting on their congruence quotients. 
The fact that the diameter is logarithmic in the size of the space in all of those cases follows from arithmetic property $\tau$ or spectral gap. 
It is expected that \cref{thm:average case} will generalize very broadly even when the space is not symmetric. 
Stated in non-precise terms, we expect that the average distance between points will agree with the lower bound given by the pigeon-hole principle. 
This is closely related to automorphic forms and approximations to the generalized Ramanujan conjecture, and we refer to \cite{golubev2023sarnak} for further discussion,
where the term {\em optimal lifting property} is used for the conjectured generalisations of \cref{thm:average case}.

The diameter is less understood, even conjecturally, and its relation with spectral questions is less direct. 
One can look at this problem as a discrete version of the Diophantine exponents of Ghosh, Gorodnik and Nevo \cite{ghosh2015diophantine}.

In the last section of this work, we consider the action of $\SL_{n}(\Z)$ on 
the $n$-dimensional affine space $A_{q}$ and the $(n-1)$-dimensional projective space $P_{q}$. 
We state some observations, conjectures, and questions. Our most interesting observation is that for $n=2$ and some non-prime $q$, the diameter of $P_q$ is as large as it can be, which is twice the almost-diameter (at least under Selberg's conjecture). 

Another interesting case is the case of division algebras, and quaternion algebras in particular. Applying the question to the LPS construction (and changing $\R$ to $\Q_{p}$), one gets the well-known problem of bounding the diameter of LPS Ramanujan graphs, where it is known that it is at least $4/3$ times the pigeon-hole lower bound (at least for some of the graphs), and it is conjectured that this bound is precise \cite{sardari2019diameter}. Our results about the projective space should be compared with the numerical study of Rivin and Sardari for quotients of LPS graphs \cite{rivin2019quantum}, where they consider only the prime case and conjecture that the diameter is the minimal possible one.

By exchanging the roles of the reals and of the finite adeles in the problem considered in this paper, one reaches the problem of approximating real matrices of determinant $1$ by integral matrices of determinant $q$, appropriately normalized by multiplying with $q^{-1/n}$. This problem was considered for $n=2$ by Tijdeman \cite{Tijdeman1986approximation} and for general $n$ by Harman \cite{harman1990approximation}. The results for $n=2$ were improved by Sardari \cite{sardari2019optimal}.

The pigeon-hole bound says that to reach every (or almost every) $\epsilon$-neighborhood of a real matrix of bounded size, one needs at least $q>\epsilon^{-(n+1)(1+o(1))}$. Given the analogy with optimal lifting, this may also be sufficient for almost all matrices, but this is unknown even for $n=2$. Harman showed that for some values of $q$, the identity matrix cannot be approximated unless $q> \epsilon^{-2n(1+o(1))}$ \cite[Theorem 1]{harman1990approximation}. If taking $q$ of this size is also enough for approximating all real matrices, then the ratio between the exponents in the average case and the worst case is $2n/(n+1)$, the same one as in the problem studied in this paper. The best known bounds for approximating all matrices are that $q> \epsilon^{-6(1+o(1))}$ is enough for $n=2$ by Sardari (\cite[Corollary 1.1]{sardari2019optimal}) and that
 $q> \epsilon^{-4n(1+o(1))}$ is enough for $n\ge 3$ by Harman \cite[Theorem 3]{harman1990approximation}.

Another similar problem is the covering exponent of $S^3$. For this problem, the almost-covering exponent is known to agree with the pigeon-hole estimate (giving rise to Golden Gates), and we expect that the covering exponent will be $4/3$ of the pigeon-hole estimate (see \cite{browning2019twisted} and the references therein). 
This is the same $4/3$ as in the $n=2$ case of the problem studied in this paper.

\subsection{Proof outlines}\label{sec:outline}

The main contribution of the paper is \cref{thm:lower bound}.
We exhibit a suitable element of $\SL_n(\Z/q\Z)$ such that any lift of it to
$\SL_n(\Z/q^2\Z)$ must already require entries of size $q^{2+o(1)}$.
This is based on the fact that given $x\in\SL_n(\Z/q\Z)$, the entries of its lifts to
$\SL_n(\Z/q^2\Z)$ must satisfy a(n inhomogeneous) linear equation $\mod q^2$
given by the tangent space to $x$ in $\SL_n$.
We will choose $x$ to be a suitable diagonal matrix so that the resulting linear
equation has small coefficients but a large inhomogeneous term, which forces at least
one entry of any lift of $x$ to be large.
Using these ideas, we can prove the following result.

\begin{prop} \label{prop:n-th root lemma}
	Assume that for $q\in \Z_{>0}$, there is $\beta\in(\Z/q^2\Z)^\times$ with $|n\beta|\ge K^{-1}q^{2}$
	and $|\beta^{n}|\le K$ for some $K\in\R_{>0}$. Then there is a diagonal $x\in \SL_n(\Z/q\Z)$ such that every lift $\gamma$ of it to $\SL_n(\Z)$ satisfies $\|\gamma\|\ge  q^{2}/nK^2$.
\end{prop}
The absolute value in the proposition is the absolute value on $\Z/q^2\Z$ defined by $|a|=\min\{|a'|:a'\in a+q^2\Z\}$.

The construction in \cref{prop:n-th root lemma} depends on finding a suitable
$\beta\in(\Z/q^2\Z)^\times$ that satisfies the hypotheses in the proposition.
Its existence is guaranteed by the following result.

\begin{theorem}\label{thm:large roots}
	For every $k>0,n>0$, there is a constant $C$ such that the following holds.
	For every $q>0$, there are $\alpha,\beta\in(\Z/q\Z)^\times$ with
	$|\alpha|<C$, $\beta^{n}=\alpha$ and $|n\beta|>C^{-1}q^{1-1/k}$.
\end{theorem}

Notice that \cref{prop:n-th root lemma} and \cref{thm:large roots} together imply \cref{thm:lower bound}.
(We need to apply \cref{thm:large roots} with $q^2$ in the role of $q$.)

The proof of \cref{thm:large roots} uses ideas from additive combinatorics, in particular \emph{Bohr Sets}.
We give some more details.
In the first part of the proof, we find a suitable collection of elements
$\beta_1,\ldots,\beta_m$ such that $|\beta_j^n|<C$ for all $j$.
Then we aim to show that $|n\beta_1^{a_1}\cdots\beta_m^{a_m}|$ is large for some choice
of $0\le a_1,\ldots,a_m\le n-1$.
In that case, $\beta=\beta_1^{a_1}\cdots\beta_m^{a_m}$ can be taken in the theorem.

With this in mind, we suppose to the contrary that the Bohr set
\begin{align*}
B(\{\beta_1^{a_1}&\cdots\beta_m^{a_m}\}, q^{-1/k})\\
&:=\{x\in\Z/q\Z:|x\cdot \beta_1^{a_1}\cdots\beta_m^{a_m}|<q^{1-1/k}\text{ for all choices of $a_j$}\}
\end{align*}
contains $n$, and hence crucially it is non-trivial.
Then we show that this Bohr set can be approximated by a generalized arithmetic
progression of rank $r \le k$.
The fact that the rank $r$ can be bounded independently of $m$ is a non-standard
feature that we have thanks to our choice of the radius, and this is of critical
importance for the success of our argument.

We observe that the set of frequencies $\{\beta_1^{a_1}\cdots\beta_m^{a_m}\}$
is ``almost-invariant" under multiplication by each $\beta_j$.
We use this to show that multiplication by $\beta_j$ induces a linear ``almost-action'' on the
generalized arithmetic progression approximating the Bohr set,
which implies that each $\beta_j$ corresponds to some matrix $B_j\in M_r(\Z)$.
We will see that the $B_j$ commute and their $n$-th powers are scalar matrices.
This will lead to a contradiction, provided $m$ is sufficiently large and the
$\beta_j$ are genuinely different in a suitable sense.

For the proof of \cref{thm:upper bound}, we generalize the strategy employed by Sarnak in \cite{sarnak2015lettermiller}. 
It consists of two steps.
\begin{enumerate}
\item Finding a lift of the first $n-1$ rows, so that they can be completed
to a matrix in $\SL_{n}(\Z)$. This is slightly complicated but can be done at a relatively small
cost, with coordinates bounded by $Cq\log q$. 
\item Finding a lift of the last row. This is simple but requires taking coordinates of size
$Cq^{2}\log q$.
\end{enumerate}

The proof in the $n=2$ case is almost identical to that in \cite{sarnak2015lettermiller}
except that we use a more elementary approach which also yields a marginally better
bound.

The proof of \cref{thm:upper bound} shows that every element in $\SL_n(\Z/q\Z)$ can be lifted to a matrix in $\SL_n(\Z)$ such that the entries of the first $n-1$ rows are bounded by $Cq\log q$ and the entries of the last row are bounded by $Cq^2 \log q$. 
This is almost optimal in the sense that one cannot hope to improve the exponent
$2$ in the bound for the entries in the last row without increasing the bound
for the first $n-1$ rows.
This is demonstrated by the following lemma.

\begin{lemma}\label{lem:skewed counting}
Let $T\in\R_{>0}$ be a number.
The number of elements in $\SL_n(\Z)$ such that the entries in the first $n-1$ rows are bounded by $T$ and the entries in the last row are bounded by $T^2$ is bounded by $CT^{n^2-1}\log T$, where $C$ is a constant depending only on $n$.
\end{lemma}

This lemma may be well known to experts and it is unlikely to be optimal, in particular,
the $\log T$ factor may be unnecessary.

\subsection{Organisation of the paper}

In \cref{sec:diagonal}, we prove \cref{prop:n-th root lemma} and also recall
Sarnak's example \cite{sarnak2015lettermiller} of an element
in $\SL_2(\Z/2^m\Z)$ without small lifts in
$\SL_2(\Z)$.
This can be explained using \cref{prop:n-th root lemma}.

\cref{sec:large roots} is devoted to the proof of \cref{thm:large roots}.
This will complete the proof of \cref{thm:lower bound}
as explained in \cref{sec:outline}.

\cref{thm:upper bound} is proved in \cref{sec:lifting sln}.
We also prove \cref{lem:skewed counting} in \cref{sec:skewed counting}.

We conclude the paper in \cref{sec:affine and projective} by making some
observations and posing some problems about related questions in the settings
of the actions of $\SL_n$ on affine and projective space.

\subsection{Notation}
Throughout this work, the constants may depend implicitly on $n$, and usually we do not mention this. In \cref{sec:large roots}, the constants may also depend on $k$. 

The notation $(\cdot,\ldots, \cdot)$ is understood to mean the column vector of
the entries listed inside the brackets.

For an element $\alpha\in\Z/q\Z$, we write
$\wt \alpha$ for the unique representative of $\alpha$ in $(-q/2,q/2]\cap\Z$.
When $x$ is a vector or matrix with entries in $\Z/q\Z$, $\wt x$ denotes
the above operation applied entry-wise. 

As mentioned above, we extend the absolute value on $\Z$ to $a\in\Z/q\Z$ by $|a|=|\wt a|$.

\subsection*{Acknowledgements}
We are grateful to Peter Sarnak for helpful discussions on the subject of this paper.
We thank the referees for their careful reading of the paper and for helpful comments that
improved its presentation.

\section{Lifting Diagonal Elements}\label{sec:diagonal}

The purpose of this section is to prove \cref{prop:n-th root lemma}.
It follows immediately from the following lemma.

\begin{lemma}\label{lem:lifting diagonal}
Let $x=\operatorname{diag}(\beta \alpha^{-1},\beta,\ldots,\beta)\in\SL_n(\Z/q^{2}\Z)$
for some $\beta,\alpha\in\Z/q^{2}\Z$. If $\gamma\in\SL_n(\Z)$ is such
that $\gamma\equiv x\mod q$, then
\[
\alpha a_{1}+a_{2}+\ldots+a_{n}\equiv n\beta \mod q^2,
\]
where $a_{1},\ldots,a_{n}$ are the diagonal entries of $\gamma$.

In particular,
\[
\n{\gamma}\ge\left|n\beta\right|/n|\alpha|.
\]
\end{lemma}

This can be deduced from the fact that the preimage of
an element $y\in \SL_n(\Z/q\Z)$ in $\SL_n(\Z/q^2\Z)$ can be identified with
$y'+qy\cdot\mathfrak{sl}_n(\Z/q\Z)$, where $y'$ is an arbitrary lift.
However, we give an elementary proof for the reader's convenience.

First, we record a trivial fact in the following lemma.

\begin{lemma}\label{lem:trivial lemma}
If $\left|a_{1}\alpha_{1}+\ldots+a_{n}\alpha_{n}\right|\ge T$ for some
$a_j,\alpha_j\in\Z/Q\Z$ and some $T,Q\in\Z_{>0}$, then
\[
\max\{\left|a_{1}\right|,\ldots,|a_{n}|\}\ge\frac{T}{n\max\{|\alpha_{1}|,\ldots,\left|\alpha_{n}\right|\}}.
\]
\end{lemma}

\begin{proof}[Proof of \cref{lem:lifting diagonal}]
Write
\[
a_{1}\equiv\beta \alpha^{-1}+qb_{1},\quad a_{i}\equiv\beta+qb_{i}\mod q^2,
\]
for $2\le i\le n$.
Then $\det\gamma=1$ implies 
\begin{align*}
1 & \equiv(\beta \alpha^{-1}+qb_{1})\prod_{i=2}^{n}(\beta+qb_{i})\mod q^2\\
 & \equiv\beta^{n}\alpha^{-1}+q\left(b_{1}\beta^{n-1}+\sum_{i=2}^{n}b_{i}\beta^{n-1}\alpha^{-1}\right)\mod q^{2}
\end{align*}

This implies 
\[
\alpha b_{1}+b_{2}+\ldots+b_{n}\equiv0\mod q,
\]
and 
\begin{align*}
\alpha a_{1}+a_{2}+\ldots+a_{n} & \equiv \alpha\beta \alpha^{-1}+(n-1)\beta+q\left(\alpha b_{1}+b_{2}+\ldots+b_{n}\right)\\
 & \equiv n\beta\mod q^{2}.
\end{align*}

The last part follows from \cref{lem:trivial lemma}.
\end{proof}

\subsection{Sarnak's example}\label{sec:sarnak expample}
We discuss Sarnak's example in \cite{sarnak2015lettermiller} for an element of $\SL_2(\Z/2^k\Z)$ without
small lifts in $\SL_2(\Z)$, which we mentioned in the introduction.

Let $n=2$, and let $q=8m$ for some $m\in\Z_{>0}$.
(Sarnak takes $m$ as a power of $2$, but this works in general.) 
Consider the element
\[
x=\left(\begin{array}{cc}
	1-4m+16m^2\\
	& 1+4m
\end{array}\right)\in\SL_2(\Z/q^2\Z).
\]
The fact that $x$ has determinant $1$ over $\Z/q^2\Z$ follows
from $q^2|(4m)^3$.
We will show that $\gamma\equiv x\mod q$ for $\gamma\in\SL_2(\Z)$
implies
\[
\tr(\gamma)\equiv \tr(x)\equiv 2+16 m^2\mod q^2.
\]
We see from this immediately that any lift of $x\mod q$ in $\SL_2(\Z)$ must
have some entries of size at least $8m^2= q^2/8$.

The claim we made is closely related to \cref{lem:lifting diagonal}, and indeed it
can be even deduced from that lemma.
Taking $\beta=1+4m\in\Z/q^2\Z$, we note that $\alpha:=\beta^2=1+8m+16m^2$ and
$\beta \alpha^{-1}=\beta^{-1}=1- 4m +16m^2$.
Now \cref{lem:lifting diagonal} gives us
\[
\alpha a_1+a_2\equiv 2+ 8m \mod q^2,
\]
where $a_1$ and $a_2$ are the diagonal entries of $\gamma$.
We observe that
\[
(\alpha-1)a_1\equiv (8m+16m^2)\cdot (1- 4m) \equiv 8m -16 m^2 \mod q^2.
\]
(Here we used $q|8m+16m^2$, and $(1-4m)\equiv a_1\mod q$.)
Taking the difference of the last two
congruences, we get $\tr(\gamma)\equiv2+16m^2\mod q^2$ as claimed.

One may also generalize this argument to $n\ge3$ by considering $q$ of the form $q=n^{3}m$, but we omit the details.

\section{Finding Large n-th Roots}\label{sec:large roots}

In this section, we prove \cref{thm:large roots}. Throughout this section, $n$ and $k$ are fixed, and all the constants may depend on them.
The proof splits into 2 cases. 
In the main case, we need a sufficient supply of small primes $p$ not dividing $q$, and also want sufficiently many $n$-th powers in $\Z/q\Z$. 
Both conditions are problematic if $q$ is divisible by many primes. Luckily, when the conditions fail, there is a different method, which we discuss first. 

Let $P$ be the set of primes with $\gcd(p-1,n)>1$ and $p\nmid n$.
(For example, if $n$ is even, $P$ is the set of all odd primes not dividing $n$). 
The set $P$ is the set of primes not dividing $n$, and such that there is a non-trivial $n$-th root of unity in $(\Z/p^m\Z)^\times$ for every $m>0$. We consider $-1$ to be non-trivial here.

We say that $q$ has a \emph{small $P$ factor} if, in the decomposition into primes $q=p_{1}^{m_{1}}\cdot\ldots\cdot p_{l}^{m_{l}}$, there is $p_{j}\in P$ with $p_{j}^{m_{j}}<q^{1/k}$. 
\begin{prop}\label{prop:small P factor}
Assume that $q$ has a small $P$ factor. Then there is $\beta\in(\Z/q\Z)^{\times}$
such that $\beta^{n}=1$ and $\left|n\beta\right|>q^{1-1/k}-n$. 
\end{prop}

\begin{proof}
Without loss of generality, we assume that $p_{1}^{m_{1}}<q^{1/k}$ and
$p_{1}\in P$. Denote $q'=q/p_{1}^{m_{1}}$ and notice that $q'> q^{1-1/k}$. 

Let $1\ne a\in\Z/p_{1}^{m_{1}}\Z$ be such that $a^{n}\equiv1\mod p_{1}^{m_{1}}$.
Using the Chinese remainder theorem, choose $\beta\in\left(\Z/q\Z\right)^{\times}$
such that 
\begin{align*}
\beta & \equiv1\mod q'\\
\beta & \equiv a\mod p_{1}^{m_{1}}
\end{align*}

Then $\beta^{n}\equiv1\mod q$, and we can write $n\beta=n+a'q'$
for some $|a'|\le p_{1}^{m_{1}}/2$.
The important fact is that $a'\ne0$, because $n\beta\not\equiv n\mod p^{m_1}$.
Then $|n\beta|\ge\left|n+a'q'\right|\ge q'-n> q^{1-1/k}-n$.
\end{proof}

\cref{prop:small P factor} handles the case when $q$ has a small $P$ factor.
From now on we assume that $q$ has no small $P$ factors.
We begin with an outline of the argument.
The first step is to find a supply of small $n$-th powers in $\Z/q\Z$,
which is the purpose of the next lemma.

\begin{lemma}
\label{lem:Enough residues}For every $m$, there is $C>0$ such that the following holds.
Let $q\in \Z_{>C}$ be without small P factors.
Then there
are $\alpha_{1},\ldots,\alpha_{m}\in(\Z/q\Z)^{\times}$ such that each
$\alpha_{i}$ has an $n$-th root $\beta_{i}\in(\Z/q\Z)^\times$, (i.e., $\beta_{i}^{n}=\alpha_{i}$),
$0<\wt\alpha_{i}<C$ and  $(\wt\alpha_{i}/\wt\alpha_{j})^{1/n}\notin\Q$ for $i\ne j$.
\end{lemma}

Recall that $\wt\alpha$ denotes the smallest (in absolute value)
integer lift of an element $\alpha\in\Z/q\Z$.

We defer the proof to \cref{sec:n powers} and continue with our outline.
Our strategy is to show that the set
\begin{equation}\label{eq:set A}
\{\beta_1^{a_1}\cdots\beta_m^{a_m}: 0\le a_i<n\}
\end{equation}
contains an element $\beta$ such that $|n\beta|$ is large, if $m$ is chosen sufficiently big in terms of $n$ and $k$,
and where $\beta_1,\ldots,\beta_m$ are the elements from the previous lemma.
Clearly, all elements of \eqref{eq:set A} satisfy that $\beta^n$ is small, so we will be done.

To achieve our aim, we need the notion of Bohr sets.
For an arbitrary set $A\subset\Z/q\Z$ of frequencies and for $\rho\in[0,1]$,
we write
\[
B(A,\rho)=\{x\in\Z/q\Z: |xa|<\rho q\text{ for all $a\in A$}\},
\] 
and call it a Bohr set.

We will consider the Bohr set $B(A,\rho)$ with the set \eqref{eq:set A} in the role
of $A$ and with $\rho= q^{-1/k}$.
A naive heuristic argument suggests that such a Bohr set should be of size
approximately $\rho^{|A|}q$, so $B(A,\rho)$ should be trivial, that is equal to
$\{0\}$, provided $A$
is sufficiently large.
We will show that this is indeed the case for our choice of $A$ if $m$
is chosen large enough.
This will finish our proof because $n\notin B(A,\rho)$ means that there is some
$a\in A$ with $|na|>\rho q= q^{1-1/k}$.

It is well known that Bohr sets are closely approximated by generalized arithmetic progressions (\cite[\S~4.22]{tao2006additive}). 
We will need the following version.

\begin{lemma}[Bohr set lemma]
\label{lem:Bohr Set Lemma}
For each $k\in\Z_{>0}$ and $c>0$, there is a constant $C$ such that 
the following holds.
Let $A\subset\Z/q\Z$ be such that $\gcd(A\cup\{q\})=1$
and let $\rho<C^{-1}q^{-1/(k+1)}$.
Then there is $C^{-1}\rho<\rho'\le\rho$,
$r\le k$ and $u_{1},\ldots,u_{r}\in\Z/q\Z$, $l_{1},\ldots,l_{r}\in\Z_{>0}$
such that
\begin{enumerate}
\item $|u_{i}a|\le c\rho'q$ for all $a\in A$, i.e., $u_{i}\in B(A,c\rho')$,
\item $B(A,\rho')\subset\left\{ \sum_{i=1}^{r}n_{i}u_{i}:|n_{i}|\le l_{i}\right\} \subset B(A,C\rho')$,
\item if $\left|n_{i}\right|\le C^{-1}l_{i}/\rho'$ and $\sum_{i=1}^{r}n_{i}u_{i}\in B(A,t\rho')$
for some $1\le t\le{\rho'}^{-1}/2$, then $\left|n_{i}\right|\le Ctl_{i}$, and
\item the map $(n_{1},\ldots,n_{r})\to\sum_{i=1}^{r}n_{i}u_{i}$ is injective
for $\left|n_{i}\right|\le C^{-1}l_{i}/\rho'$.
\end{enumerate}
\end{lemma}

The proof of this lemma follows \cite{tao2006additive},
and we defer it to \cref{sec:Bohr},
but we point out an important feature of the above formulation.
Namely, that we can bound $r$, the rank of the generalized arithmetic progression,
independently of the size of the set of frequencies $A$.
This is possible, because of our assumption that the radius $\rho$ is very small.
This feature will be important for us because our argument will require us to choose
the set $A$ to be sufficiently large in terms of (the upper bound we have for) $r$.

We observe that the set \eqref{eq:set A}, which we use in the role of $A$,
is ``almost invariant" under multiplication by each
$\beta_i$, and we will show in the next lemma that this induces a dual ``almost-action"
on the Bohr set.

\begin{lemma}\label{lem:lifting beta}
For all $D_1,D_2\in \R_{>0}$, there is $C\in \R_{>0}$ such that the following holds.
Let $A\subset\Z/q\Z$ such that $\gcd(A\cup\{q\})=1$, and let $\alpha,\beta\in\Z/q\Z$
such that $\beta A\subset A\cup \alpha A$ and $|\alpha|\le D_1$.
Let $\rho\in(0,C^{-1})$, $r\in\Z_{>0}$, $u_1,\ldots,u_r\in\Z/q\Z$
and $l_1,\ldots, l_r\in\Z_{>0}$ be such that
\begin{enumerate}
\item $u_{i}\in B(A,D_1^{-1}\rho)$,
\item $B(A,\rho)\subset\left\{ \sum_{i=1}^{r}n_{i}u_{i}:|n_{i}|\le l_{i}\right\} \subset B(A,D_2\rho)$, and
\item If $\left|n_{i}\right|\le D_2^{-1}l_{i}/\rho$ and $\sum_{i=1}^{r}n_{i}u_{i}\in B(A,t\rho)$
for some $1\le t\le{\rho}^{-1}/2$, then $\left|n_{i}\right|\le D_2tl_{i}$.
\end{enumerate}

Then there is a matrix $B=(B_{i,j})\in M_r(\Z)$ such that
\[
\beta (u_1,\ldots, u_r)\equiv B(u_1,\ldots,u_r) \mod q
\]
and $|B_{i,j}|\le C l_j/l_i$ for all $i$, $j$.
(Recall that the notation $(u_1,\ldots,u_r)$ is understood as a column vector.)
\end{lemma} 

We postpone the proof of this lemma to \cref{sec:lifting beta},
and now we finish the proof of \cref{thm:large roots} in a few strokes.

\begin{proof}[Proof of \cref{thm:large roots}]
By \cref{prop:small P factor}, we may assume that $q$ has no small $P$-factors. 
	
Let $m=n^{k}+1$ (the reason for this choice will be explained right at the end of the proof), and let $C_1$,
$\alpha_{1},\ldots,\alpha_{m}$ be as in \cref{lem:Enough residues}, i.e., $0<\wt\alpha_{s}<C_1$. Let $\beta_{s}$ be some $n$-th root of $\alpha_{s}$. 

Choose $c=C_1^{-1}$ in \cref{lem:Bohr Set Lemma}, and let $C_2$ be the constant
$C$ in that lemma.
	
Consider the set 
\[
A=\left\{ \beta_{1}^{a_{1}}\cdot\ldots\cdot\beta_{m}^{a_{m}}:0\le a_{s}<n\right\} .
\]
Let $\rho=q^{1/k}$. 
We suppose to the contrary that the theorem is false.
As we already discussed, this implies that $\lvert n a\rvert \le C_2^{-1}\rho q$ for all $a\in A$,
and hence $n\in B(A,\rho')$ for all $\rho'>C^{-1}\rho$.
In particular, we get $r\ge 1$ when we apply \cref{lem:Bohr Set Lemma}
for $A$ and $\rho$.

Note that $\rho<C_2^{-2}q^{-1/(k+1)}$ if $q$ is sufficiently large.
By \cref{lem:Bohr Set Lemma}, there is $C_2^{-1}\rho<\rho'\le\rho$, $0<r\le k$,
$u_{1},\ldots,u_{r}\in \Z/q\Z$ and $l_{1},\ldots,l_{r}\in \Z_{>0}$ such that 
items (1)--(4) in the lemma hold.
We can now apply \cref{lem:lifting beta} for each pair $\alpha_s, \beta_s$ in the
role of $\alpha$ and $\beta$ and with $\rho'$ in the role of $\rho$, $C_1$ in the role of $D_1$
and $\max(C_1,C_2)$
in the role of $D_2$.
Denoting by $C_3$ the constant $C$ in \cref{lem:lifting beta}, we get
matrices $B^{(s)} \in M_n(\Z)$ with
\[
\beta_s(u_1,\ldots, u_r)\equiv B^{(s)}(u_1,\ldots, u_r) \mod q
\]
and $|B^{(s)}_{ij}|\le C_3l_j/l_i$.

We observe that
\begin{align*}
B^{(s_1)}B^{(s_2)}(u_1,\ldots, u_r)&\equiv\beta_{s_1}\beta_{s_2}(u_1,\ldots, u_r) \\
&\equiv B^{(s_2)}B^{(s_1)}(u_1,\ldots, u_r) \mod q.    
\end{align*}
The $i,j$ entries of both matrices $B^{(s_1)}B^{(s_2)}$ and $B^{(s_2)}B^{(s_1)}$
are bounded by $rC_3^2l_j/l_i\le C_2^{-1} l_j/\rho'$ provided $q$ is sufficiently large.
By the injectivity property in item (4) of \cref{lem:Bohr Set Lemma},
we get $B^{(s_1)}B^{(s_2)}=B^{(s_2)}B^{(s_1)}$.
A very similar argument gives $(B^{(s)})^n=\wt\alpha_sI$.

We finally get $m=n^{k}+1$ commuting matrices $B^{(1)},\ldots,B^{(n^{k}+1)}\in M_{r}(\Z)$, $r\le k$, such that $(B^{(s)})^{n}=\wt\alpha_{s}I$. 
Since the minimal polynomial of $B^{(s)}$ has no double roots, 
$B^{(s)}$ is diagonalizable, and since all the matrices commute, 
there is a matrix $P \in M_r(\C)$ such that all the matrices $PB^{(s)}P^{-1}$ are diagonal. 
Moreover, the diagonal of $PB^{(s)}P^{-1}$ is of the form $\wt\alpha_{s}^{1/n}(\zeta_{n}^{t_{s1}},\ldots,\zeta_{n}^{t_{sr}})$,
for $\wt\alpha_{s}^{1/n}\in\R_{>0}$, $\zeta_{n}=\exp(2\pi i/n)$ and $0\le t_{sj}\le n-1$. 
By the pigeon-hole principle, there are $s_{1}\neq s_{2}$ with $t_{s_{1}j}=t_{s_{2}j}$ for all $j$. 
This implies that $B^{(s_{1})}$ is a scalar multiple of $B^{(s_{2})}$, that is, $B^{(s_{1})}=\left(\wt\alpha_{s_1}/\wt\alpha_{s_2}\right)^{1/n}B^{(s_2)}$.
This is a contradiction since those are integer matrices and $\left(\wt\alpha_{s_1}/\wt\alpha_{s_2}\right)^{1/n}\notin\Q$.
\end{proof}

\subsection{Finding small \texorpdfstring{$n$}{n}-th powers}\label{sec:n powers}
In this section, we prove \cref{lem:Enough residues}.
We begin by recording a simple consequence of $q$ not having small $P$ factors.

\begin{lemma}\label{lem:small primes exist}
For every $m>0$, there is $C>0$ such that if $q$ has
no small $P$ factors, there are at least $m$ different primes $p\nmid q$, $p\in P$ 
such that $p<C$. 
\end{lemma}
\begin{proof}
By Dirichlet's theorem, there are infinitely many primes $p$ with $\gcd(p-1,n)>1$ and $p\nmid n$.
These all belong to $P$.
We take the first $m+k$ of them, and let $C$ be an upper bound for them.
By definition, $q$ having no small $P$ factors implies that there are at most
$k$ primes in $P$ that divide $q$.
The remaining $m$ primes satisfy the conclusion of the lemma.
\end{proof}

\begin{proof}[Proof of \cref{lem:Enough residues}]
There is a group homomorphism $f\colon \left(\Z/q\Z\right)^{\times}\to\left(\Z/q\Z\right)^{\times}/\left(\Z/q\Z\right)^{\times n}$,
where $(\cdot)^{\times n}$ denotes the set of invertible
$n$-th powers in a commutative ring.
We claim that under the given conditions, the image of $f$ is of bounded size $C\le(2n)^{\omega(n)+k}$, where $\omega(n)$ is the number of distinct prime factors of $n$. Indeed, if $q=\prod p_{i}^{m_{i}}$,
then 
\[
\left(\Z/q\Z\right)^{\times}/\left(\Z/q\Z\right)^{\times n}\cong\prod\left(\left(\Z/p_{i}^{m_{i}}\Z\right)^{\times}/\left(\Z/p_{i}^{m_{i}}\Z\right)^{\times n}\right).
\]
At most $\omega(n)+k$ of the $p_{i}$-s are either dividing $n$ or $\gcd(p_i-1,n)>1$.
Since $(\Z/p^{m}\Z)^{\times}$ is either cyclic or $\Z/2\Z$ times
a cyclic group (if $p=2$, $m\ge3$), for those primes $|(\Z/p_{i}^{m_{i}}\Z)^{\times}/(\Z/p_{i}^{m_{i}}\Z)^{\times n}|\le2n$.
For all the other primes, since $\gcd((p-1)p,n)=1$,  $|\left(\Z/p_{i}^{m_{i}}\Z\right)^{\times}/\left(\Z/p_{i}^{m_{i}}\Z\right)^{\times n}|=1$.
This proves our claim about the size of the image of $f$. 

Take $m+C$ primes $p_{1},\ldots,p_{m+C}$ as in \cref{lem:small primes exist}. This can be done with each prime bounded by a constant depending on $m$ (and also $n$ and $k$). Choose for each element $a$ in $f(\{p_{1},\ldots,p_{C+m}\})$ an element $p_{a}$ such that $f(p_{a})=a$. There are at least $m$ remaining primes.
For each of the $m$ remaining primes $p_{j}$ there is $p_{a}$ with $f(p_{j})=f(p_{a})$.
But then $\alpha_{j}=p_{j}p_{a}^{n-1}$ satisfies $f(\alpha_{j})=1$, and the $\alpha_{j}$ are the required elements. 
\end{proof}

\subsection{Bohr sets}\label{sec:Bohr}
We turn to the proof of \cref{lem:Bohr Set Lemma}. We will need the following lemma from the geometry of numbers, which is a combination of the theorems of Mahler and John.
\begin{lemma}[Discrete John's theorem]
\label{lem:Discrete John Theorem}
For every $r\in\Z_{>0}$, there is a constant $D>0$ such that the following holds.
Let $\Gamma \subset \R^N$ be a lattice of rank $r$, and let
$B\subset\R^{N}$ be convex and symmetric. Suppose $\dim\spann(\Gamma\cap B)=r$.
Then there is a basis $w_{1},\ldots,w_{r}$ of $\Gamma$ and $l_{1},\ldots,l_{r}\in\Z_{>0}$
such that for every $\alpha\ge1$, 
\[
\left\{ \sum n_{i}w_{i}:\left|n_{i}\right|\le \alpha l_{i}\right\} \subset\Gamma\cap \alpha DB,
\]
and 
\[
\Gamma\cap\alpha B\subset\left\{ \sum n_{i}w_{i}:\left|n_{i}\right|\le \alpha Dl_{i}\right\} .
\]
\end{lemma}
The proof is essentially the same as the proof of \cite[Lemma~3.36]{tao2006additive}.

\begin{proof}[Proof of \cref{lem:Bohr Set Lemma}]
We follow the proof of \cite[Lemma~4.22]{tao2006additive}.
We define a lattice in $\R^{A}$ by 
\[
\Gamma=\Z^{A}+\Big\{\left(\frac{x a}{q}\right)_{a\in A}:x\in\Z/q\Z\Big\}.
\]

We have a map 
\[
\Phi\colon \Z/q\Z\to\Gamma/\Z^{A}
\]
given by 
\[
\Phi(x)=\left(\frac{xa}{q}\right)_{a\in A}+\Z^{A}.
\]
By definition, the map is onto and the kernel is
\[
\left\{ b\in\Z/q\Z:ba=0\text{ for all $a\in A$}\right\}=\{0\}
\]
since $\gcd(A,q)=1$.
So $\Z^{A}$ is a subgroup of index $q$ in $\Gamma$. 

Let $K_{\rho}=\left\{ x\in\R^{A}:\n x_{\infty}<\rho\right\}$. Then
by definition,
\[
B(A,\rho)=\Phi^{-1}(K_{\rho}\cap\Gamma).
\]

We now claim that for $\rho<1/(2(k+1)q^{1/(k+1)})$,
$\dim\spann (K_{\rho}\cap\Gamma)\le k$.
Indeed, if we have linearly independent $v_{1},\ldots,v_{k+1}\in K_{\rho}\cap\Gamma$,
then also 
\[
\sum m_{i}v_{i}\in K_{1/2}\cap\Gamma
\]
for $|m_{i}|<\rho^{-1}/2(k+1)$, and by linear independence, all the elements are different. 
For every $i$, the number of such $m_{i}$ is $2\lfloor\rho^{-1}/2(k+1)\rfloor+1\ge\rho^{-1}/2(k+1)$,
so the total number of elements in $K_{1/2}\cap\Gamma$ is at least $(\rho^{-1}/2(k+1))^{k+1}>q$. 
On the other hand, every two elements $\gamma_1\ne \gamma_2 \in \Gamma \cap K_{1/2}$ belong to different cosets of $\Gamma / \Z^A$, so the index of $\Z^A$ in $\Gamma$ is more than $q$. 
This is a contradiction.

Let $D$ be an upper bound for the constant denoted by the same letter in
\cref{lem:Discrete John Theorem} when it is applied for some $r\le k$.
Since $\dim \spann (K_{\rho'}\cap\Gamma) \le k$, if $C\ge\left(c^{-1}D\right)^{k}$, there is $C^{-1}\rho<\rho'\le\rho$ such that $\spann (K_{cD^{-1}\rho'}\cap\Gamma)=\spann (K_{\rho'}\cap\Gamma)$.
Denote $V=\spann(K_{\rho'}\cap\Gamma)$, and then $\Gamma\cap V$ is a lattice in $V$. 

We apply \cref{lem:Discrete John Theorem} with $B=K_{cD^{-1}\rho'}$ and
$\Gamma\cap V$ in the role of $\Gamma$.
Then there is a basis $w_{1},\ldots,w_{r}\in \Gamma$ such that
\begin{enumerate}
\item $w_i\in K_{c\rho'}$ for all $i$,
\item[(2.i)] each element in $K_{\rho'}\cap\Gamma$ can be written as $\sum n_{i}w_{i}$ with $\left|n_{i}\right|\le c^{-1}D^2l_{i}$,
\item[(2.ii)] $\sum n_{i}w_{i}\in K_{D^2\rho'}\cap\Gamma$ for $\left|n_{i}\right|\le c^{-1}D^2l_{i}$,
\item[(3)] if $\sum n_{i}w_{i}\in K_{t\rho'}$ for some $1\le t\le {\rho'}^{-1}/2$, then $\left|n_{i}\right|\le c^{-1}D^2 t l_i$, and
\item[(4)] $\sum n_{i}w_{i}\in K_{1/2}\cap\Gamma$ for $\left|n_{i}\right| \le (1/2c){\rho'}^{-1}l_i$.
\end{enumerate}

Now we take $u_{i}=\Phi^{-1}(w_{i}+\Z^A)$ and replace $c^{-1}D^2l_i$ by $l_i$.
We get the required properties.
To deduce the injectivity property in item (4), we use the fact that $K_{1/2}$
contains at most one element from each coset of $\Z^A$.
\end{proof}

\subsection{Proof of \texorpdfstring{\cref{lem:lifting beta}}{Lemma 3.4}}\label{sec:lifting beta}
We claim that $\beta u_{i}\in B(A,\rho')$ for all $i$.
Indeed, if $a\in A$ then either $\beta a\in A$, or $\beta a=\alpha a'$ for some $a'\in A$. 
Using item (1), in the first case
\[
\left|u_{i}\beta a\right|<D_1^{-1}\rho q<\rho q,
\]
while in the second case,
\[
|u_{i}\beta a|=|\alpha u_{i} a'|\le|\alpha| |u_{i}a'|<|\alpha|\cdot D_1^{-1}\rho q\le \rho q.
\]
So $\beta u_{i}\in B(A,\rho)$ in either case. 

Using the first inclusion in item (2), we find a matrix $B\in M_r(\Z)$ with
$|B_{i,j}|\le l_j$ and
\[
\beta (u_1,\ldots, u_r)\equiv B(u_1,\ldots, u_r) \mod q.
\]

It is left to prove an improved bound for the entries.
To this end, fix some index $i$.
We notice that for $|n|\le l_{i}$, $n u_{i}\in B(A,D_2\rho)$
by the second inclusion in item (2).
The same argument as above yields that $\beta n u_{i} \in B(A, D_2^2\rho)$.

Now we prove by induction that $n |B_{ij}|\le D_2^3 l_j$ for $n=1,\ldots, l_i$.
We have already seen this for $n=1$, so suppose $n>1$ and the claim holds for $n-1$.
Then
\[
n |B_{i,j}|\le D_2^3 l_j+|B_{ij}|\le (D_2^3+1) l_j\le D_2^{-1}l_j/\rho
\]
provided $C> D_2(D_2^3+1)$, which we can assume.
We also have
\[
\beta n u_{i}=\sum_{j}n B_{ij}u_{j}\in B(A,D_2^2\rho).
\]
Now we use item (3) with $t=D_2^2<\rho^{-1}/2$
and get $n|B_{ij}|\le D_2tl_j$ proving the claim for $n$.

Now our claim is proved by induction, and taking $n=l_i$, it gives
$|B_{ij}|\le D_2^3 l_j/l_i$, as required.

\section{Lifting to \texorpdfstring{$\SL_n(\Z)$}{SLn(Z)}}\label{sec:lifting sln}

We prove \cref{thm:upper bound} in this section.
We extend Sarnak's proof of the $n=2$ case in \cite{sarnak2015lettermiller}.
In the $n=2$ case, we only deviate from \cite{sarnak2015lettermiller}
in the proof of \cref{prop:step 1}, which also
yields a modest improvement of the bound.
See \cref{sec:first step} for a discussion.

The first step is to find a suitable lift of the first $n-1$ rows, which
is the content of the next statement.

We recall the notation that for an element $\alpha\in\Z/q\Z$, we write
$\wt \alpha$ for the unique representative of $\alpha$ in $(-q/2,q/2]\cap\Z$.
When $x$ is a vector or a matrix with entries in $\Z/q\Z$, $\wt x$ denotes
the above operation applied entry-wise. 

\begin{prop}\label{prop:step 1}
There is $C=C(n)>0$ such that for every $A\in M_{n-1,n}(\Z/q\Z)$ that can
be extended to an element of $\SL_n(\Z/q\Z)$, 
there exists a matrix $X\in M_{n-1,n}(\Z)$ with coordinates bounded by $C\log q$ such that $\wt A+qX$ can be extended to a matrix in $\SL_{n}(\Z)$. 
\end{prop}

We defer the proof to \cref{sec:first step}.

The next step is to complete the matrix with a last row to an element
of $\SL_n(\Z)$ while we still control the entries of the resulting matrix.

\begin{prop}\label{prop:step 2}
For every $B\in M_{n-1,n}(\Z)$ with coordinates bounded by $T$ that can be extended by adding a last row to $\SL_n(\Z)$, 
one may choose the last row to be with coordinates bounded by $nT/2+1$.
\end{prop}

The proof is deferred to \cref{sec:second step}, but we note that
the proposition is very simple in the $n=2$ case.
Indeed, given integers $b_{11},b_{12}$ with $\gcd(b_{11},b_{12})=1$, the Euclidean algorithm
yields integers $a_1,a_2$ such that $a_1b_{11}+a_2b_{12}=1$, and $|a_1| \le |b_{12}|$, $|a_2| \le |b_{11}|$. Therefore, we can
simply take $b_{21}=-a_2$ and $b_{22}=a_1$.
This proves the proposition in the $n=2$ case.

The last row of the matrix given by \cref{prop:step 2} may not
have the correct reduction $\mod q$.
To remedy this problem, we use linear algebra over $\Z/q\Z$.

\begin{lemma}\label{lem:linear algebra}
If the vectors $v_{1},\ldots,v_{n-1}, v_n\in (\Z/q\Z)^n$ are rows of a matrix
with determinant $1$, and $w\in(\Z/q\Z)^n$ is such that $\det(v_1,\ldots,v_{n-1},w)=0$,
then we can find $\alpha_{1},\ldots,\alpha_{n-1}\in \Z/q\Z$ such that $w=\sum_{i=1}^{n-1}\alpha_{i}v_{i}$.
\end{lemma}

\begin{proof}
Since the matrix composed of the rows $v_1,\ldots,v_n$ has determinant $1$, it has
an inverse, see \cite[Proposition XIII.4.16]{lang}.
Therefore, $w$ is a linear combination of the vectors  $v_1,\ldots,v_n$.
It follows from Carmer's rule (\cite[Theorem XIII.4.4]{lang}) that the coefficient
of $v_n$ must be $0$ in this linear combination.
\end{proof}

\begin{proof}[Proof of \cref{thm:upper bound}]
Given $A\in \SL_n(\Z/q\Z)$, by \cref{prop:step 1}, we may find a matrix $B\in M_{n-1,n}(\Z)$ such that $B \mod q$ agrees with $A$ on the first $n-1$ rows, the entries of $B$ are bounded by $Cq\log q$, and $B$ may be completed to a matrix in $\SL_n(\Z)$ using a vector $v_n \in \Z^n$.  
	
By \cref{prop:step 2}, we may assume that the entries of $v_n$ are bounded by $Cq\log q$.
	
To complete the proof, we need to change $v_{n}$ so that it agrees with the last row $a_n$ of $A$ modulo $q$. 
	
It holds that
$\det\left(\begin{array}{c}
A\\
a_{n}-v_{n}
\end{array}\right)\equiv0\mod q$,
which implies, by \cref{lem:linear algebra}, that there are $\alpha_1,\ldots,\alpha_{n-1}\in \Z/q\Z$ such that $a_{n}-v_{n}\equiv \sum_{i=1}^{n-1}\alpha_{i}a_{i} \mod q$, where the $a_i$ are the rows of $A$.
Writing $b_i$ for the rows of $B$,
\[
\left(\begin{array}{c}
B\\
v_{n}+\sum\wt\alpha_{i}b_{i}
\end{array}\right)
\]
is the required lift, and its entries are bounded by $Cq^2 \log q$.
\end{proof}

\subsection{Lifting the first \texorpdfstring{$n-1$}{n-1} rows}
\label{sec:first step}
We first give the proof in the $n=2$ case, which is simpler.
In that case, we need to lift two coprime elements of $\Z/q\Z$ to coprime elements
of $\Z$.
The next lemma will take care of the large primes.
For $a,b,P \in \Z$, we say that $(a,b)_{P}=1$ if whenever $p$ is a prime with $p|a$ and $p|b$ then $p|P$. 

\begin{lemma}
	\label{lem:disjointness lemma}For every $\epsilon>0$, there are $K>0$,
	$C>0$ such that the following holds.
	Let $P$ be the product of primes smaller than $K$,
	and let $a,b\in\Z$, $q\in\Z_{>0}$, $T>C\log\max\{q,|a|,2\}$.
	Then 
	\[
	\left\{ (x,y)\in\Z^{2}: (a+xq,b+yq)_{Pq}\neq1,\left|x\right|\le T,\left|y\right|\le T\right\} \le\epsilon T^{2}.
	\]
\end{lemma}

We record a simple fact, which will be used in the proof.
\begin{lemma}\label{lem:trivial lemma 2}
	Let $q$ be a natural number and $v_0\in \Z/q\Z$. Then 
	\[
	|\{x\in \Z : |x|\le T, x\equiv v_0 \mod q\}|\le \left\lceil \frac{2T+1}{q}\right\rceil.
	\]
\end{lemma}

\begin{proof}[Proof of \cref{lem:disjointness lemma}]
	By \cref{lem:trivial lemma 2}, the number of times that a prime $p\nmid q$ divides $a+xq$ for $|x|\le T$ is at most $(2T+1)/p+1$. Therefore, the number of times a prime $p\le2T+1$, $p\nmid q$ divides both $a+xq$ and $b+yq$ is at most $(2(2T+1))^2/p^2\le 36 T^{2}/p^{2}$.
	
	Now fix some $x$ such that $a+xq\ne0$.
	The number of primes dividing $a+xq$ is at most $\log_{2}(|a|+Tq)$.
	Given such a prime with $p>2T+1$ and $p\nmid q$, the number of times it divides $b+yq$ is at most $1$. 
	This gives at most $(2T+1)\log_2(|a|+Tq)$ pairs $(x,y)$ such that $\gcd(a+xq,b+yq)$
	is divisible by a prime greater than $2T+1$ not dividing $q$.
	
	Finally, we should also consider the possibility that $a+xq=0$, which happens for at most one $x$ and gives at most $2T+1$ other pairs with $\gcd(x,y)_{Pq}\neq0$.
	
	Therefore, 
	\begin{align*}
		& \left\{ x,y\in\Z^{2}: (a+xq,b+yq)_{Pq}\neq 1,\left|x\right|\le T,\left|y\right|\le T\right\} \\
		& \le36T^{2}\sum_{K<p\le2T+1}p^{-2}+(2T+1)\log_2(|a|+Tq)+(2T+1).
	\end{align*}
	By choosing $K$ sufficiently large, we can make sure that 
	\[
	36T^{2}\sum_{p>K}p^{-2}<\epsilon T^{2}/2,
	\]
	and by choosing $C$ sufficiently large and $T>C\log_2\max\{q,|a|,2\}$, we can make sure that 
	\[
	(2T+1)\log(|a|+Tq)+2T+1<\epsilon T^{2}/2.
	\]
	The claim follows.
\end{proof}

\begin{proof}[Proof of \cref{prop:step 1} in the $n=2$ case]
Note that $\gcd(\wt A_{11},\wt A_{12},q)=1$.
To prove the proposition, it is enough
to change both $\wt A_{11}$ and $\wt A_{12}$ by multiples of $q$ to make them coprime.

Let $\epsilon=0.1$, and let $K,C$ be as in \cref{lem:disjointness lemma}. 

Let $P$ be the product of primes smaller than $K$, without the divisors of $q$.
Let $a_{1}=\wt A_{11}+x_{1}q$, $b_{1}=\wt A_{12}+y_{1}q$ be such that $a_{1}\equiv b_{1}\equiv 1\mod P$.
This can be done with $0\le x_{1},y_{1}\le P-1$, and we have $\gcd(a_1,b_1,Pq)=1$.

Then we use \cref{lem:disjointness lemma} with $Pq$ in the role of $q$
to find $0\le x_{2},y_{2}\le C\log_{2}q$
such that $a_{2}=a_{1}+x_{2}Pq$, $b_{2}=b_{1}+y_{2}Pq$ satisfy
$(a_{2},b_{2})_{Pq}=1$. Since $\gcd(a_1,b_1,Pq)=1$, we must have
$\gcd(a_2,b_2,Pq)=1$ and then $\gcd(a_{2},b_{2})=1$ follows. 

To conclude the proof, we may take $X_{11}=x_1+x_2P$ and $X_{12}=y_1+y_2P$ so that
the entries of $\wt A+qX$ are $a_2$ and $b_2$.
\end{proof}

\begin{remark}\label{rmk:compare with Sarnak}
Our proof differs from the one in \cite{sarnak2015lettermiller} in that we
modify both $\wt A_{11}$ and $\wt A_{12}$ by a multiple of $q$ rather than just one of them.
Indeed, Sarnak finds an element coprime to $\wt A_{11}$ in the arithmetic progression
$\wt A_{12}+jq$ using a bound on the Jacobsthal function due to Iwaniec \cite{iwaniec1978},
which leads to a slightly weaker bound.
\end{remark}
For the general case, we need some lemmata. 
\begin{lemma}
\label{lem:primitive lattice}Let $v_{1},\ldots,v_{k}\in\Z^{n}$ for $k<n$. 
The following three conditions are equivalent:
\begin{enumerate}
\item There is a matrix in $\SL_{n}(\Z)$ with first $k$ rows $v_{1},\ldots,v_{k}$.
\item There is a rational subspace $V\subset \R^n$ such that $v_1,\ldots,v_k$ are a basis of the lattice $\Gamma = V\cap \Z^n$.
\item For every prime $p$, $v_{1}\mod p,\ldots,v_{k}\mod p$ are linearly independent as elements
of the vector space $(\Z/p\Z)^{n}$ over the finite field $\Z/p\Z$.
\end{enumerate}
\end{lemma}

This is standard, but we give a proof for the reader's convenience.

\begin{proof}
    For a proof that the second condition implies the first, see \cite[Lecture VIII \S 1]{siegel1989lectures}.
    The converse is trivial.

    To prove the equivalence of the first and third conditions, look at the $k\times n$ matrix $A=(v_{i,j})$ whose rows are $v_1,\ldots,v_k$. 
    
    Notice that both conditions are invariant by the following two sets of elementary operations, which can be seen as elementary row and column operations on the matrix $A$: 
    \begin{enumerate}
        \item For some $1\le i\ne j\le k$, $\alpha \in \Z$, replace $v_i$ by $v_i +\alpha v_j$.
        \item For all $1\le i\le k$, for some $1\le j\ne l\le n$, $\alpha \in \Z$, replace $v_{i,j}$ by $v_{i,j}+\alpha v_{i,l}$.
    \end{enumerate}
    If the $\gcd$ of the entries of $v_1$ is not $1$, both conditions obviously fail. If the $\gcd$ of the entries of $v_1$ is $1$, we may assume (by using operations of type (2)) that $v_1$ is equal to $v_1= (1,0,\ldots,0)$ and (by using operations of type (1)) that $v_{i,1}=0$ for $i>1$. 
    Repeating this procedure, we see that either both conditions fail, or we can
    transform the vectors using the above two types of operations such that $v_{i,i}=1$
    and $v_{i,j}=0$ for $i\neq j$ for all $i$ and $j$.
    In this case, both conditions are easily seen to hold.
\end{proof}

\begin{lemma}\label{lem:large primes}
There are constants $K=K(n)$ and $C=C(n)$ such that for every $A\in M_{n-1,n}(\Z/q\Z)$
there is $X\in M_{n-1,n}(\Z)$ with coordinates bounded by $C\log q$ such that $\wt A+qX$ has linearly independent rows over
any prime $p>K$ not dividing $q$.
\end{lemma}

To prove this lemma, we use a generalized version of \cref{lem:trivial lemma 2}.
\begin{lemma}\label{lem:counting mod p}
	Let $p$ be a prime, let $V\subset (\Z/p\Z)^n$ be a proper subspace, and let $v_0\in (\Z/p\Z)^{n}$.
	Then for every $T>0$,
	\begin{align*}
	\left|\{(x_1,\ldots,x_n)\in \Z^n : (x_1,\ldots,x_n) \mod p \in v_0+V, |x_i|\le T\}\right|\\
	\le (2T+1)^{n-1}\left\lceil \frac{2T+1}{p}\right\rceil.
	\end{align*}
\end{lemma}
\begin{proof}
	Let $e_1,\ldots,e_{n}$ be the standard basis of $(\Z/p\Z)^{n}$. Since $V$ is proper, one of the $e_i$ does not belong to $V$.  
	Without loss of generality, we may assume that $e_1\notin V$. Then for every choice of $x_2,\ldots,x_{n}$, by \cref{lem:trivial lemma 2}, there are at most $\left\lceil \frac{2T+1}{p}\right\rceil$ choices of $x_1$ with $(x_1,\ldots,x_n) \mod p \in v_0+V$. The result follows.
\end{proof}

\begin{proof}[Proof of \cref{lem:large primes}]
We prove by induction, with the base case following from \cref{lem:disjointness lemma} when $n=2$.

Let $A'$ be the $(n-1)\times(n-2)$ submatrix of $A$, consisting of the first $n-2$ columns. By the induction hypothesis (and exchanging rows and columns), we can choose $X'\in M_{n-1,n-2}(\Z)$ with coordinates bounded by $C\log q$
such that $\wt A'+qX'$ has linearly independent columns modulo every $p>K$ that does not divide $q$. 

We will extend $X'$ to $X$.
The rows of $\wt A+qX$ will be linearly dependent modulo a prime $p$ if and only if both columns modulo $p$ will be linear sums of the first $n-2$ columns modulo $p$.
Fix a number $T$ that is larger than all entries of $X'$ to be determined later.
We consider all $(2T+1)^{n-1}$ choices for each of the last two columns of $X$
with entries bounded by $T$, and count the number of those choices that make
both new columns of $\wt A+qX$ a linear combination of the first $n-2$ modulo some prime $p$.  

Fix a prime $p\nmid q$.
Then for either of the last two columns of $\wt A+qX$,
the number of choices that make the column
a linear combination of the first $n-2$ columns
is at most $(2T+1)^{n-2}\left\lceil(2T+1)/p\right\rceil$,
by \cref{lem:counting mod p}.
This is bounded by $(6T)^{n-1}/p$ if $p\le 2T+1$ and by $(3T)^{n-2}$ if $p>2T+1$.

For primes $p\le 2T+1$, this gives at most $(6T)^{2n-2}/p^2$ times when the rows of $\wt A+qX$ will be linearly dependent modulo $p$.

To handle the primes $p>2T+1$, we distinguish two cases.
First we count those choices such that the determinant of the first
$n-1$ columns of $\wt A+qX$ is non-zero.
For every such choice of the $(n-1)$-th column, the determinant in question is bounded by
$\le C_{n}T^{n-1}q^{n-1}$, so it has at most $\le(n-1)\log_2 q+(n-1)\log_2 T+O_{n}(1)$
prime divisors. For each such prime $p>2T+1$, the number of times the last column is a linear sum modulo $p$ is at most $(3T)^{n-2}$. This gives another
\[
(3T)^{n-1}((n-1)\log_2 q+(n-1)\log_2 T+O_{n}(1))(3T)^{n-2}
\]
choices that make the rows of $\wt A+qX$ linearly dependent modulo some prime $p>2T+1$. 

Finally, by looking at a prime larger than $2T+1$, we know by the above
that the number of choices for the $n-1$-th column such that the determinant of the first $n-1$ columns of $\wt A+qX$ is zero modulo this prime is at most $(3T)^{n-2}$.
This also gives a bound on the number of times the same determinant is zero in $\Z$.
This gives at most another 
\[
(3T)^{n-2}(3T)^{n-1}
\]
choices of $X$ that make the rows of $\wt A+qX$ linearly dependent modulo some prime $p>2T+1$. 

The total sum is 
\begin{align*}
 & (6T)^{2n-2}\sum_{p>K}\frac{1}{p^{2}}\\
+ & (3T)^{2n-3}\left((n-1)\log_2 q+(n-1)\log_2 T+O_{n}(1)\right)\\
+ & (3T)^{2n-3}.
\end{align*}
By choosing $C,K$ large enough and $T>C\log_2 q$, the above is at most $0.01\cdot T^{2n-2}$, so there is a choice of $X$ as needed.
\end{proof}

\begin{proof}[Proof of \cref{prop:step 1}]
By \cref{lem:primitive lattice}, it is enough to show that the rows
of $\wt A+qX$ are linearly independent over any prime $p$. 

This is automatic for all primes $p|q$, because the rows of $A$, and hence the rows
of $\wt A+qX$ are linearly independent modulo such primes by assumption.

Let $K$ and $C$ be as in \cref{lem:large primes}. 

Let $P$ be the product of primes smaller than $K$, without the primes
dividing $q$. By choosing $0\le X_{i,j}^{(1)}\le P-1$ suitably, we may make sure
that the rows of $\wt A+qX^{(1)}$ are linearly independent modulo each $p|P$.
(We could arrange for $\wt A+qX^{(1)}\equiv I\mod p$ for each such prime for example.)

Applying \cref{lem:large primes} with $Pq$ in the role of $q$, we find $X^{(2)}$ with entries bounded by $C\log q$
such that $\wt A+qX^{(1)}+Pq X^{(2)}$ has linearly independent rows modulo $p$
for all primes $p$.
Combining the estimates for the entries of $X^{(1)}$ and $X^{(2)}$, the proof is
complete.
\end{proof}

\subsection{Completing the matrix}\label{sec:second step}

We prove \cref{prop:step 2} in this section.

We recall that the covolume $d(\Gamma)$ of a lattice $\Gamma \subset \R^n$ of rank $k$ is
defined as the volume of the fundamental domain of $\Gamma$ in the subspace $V= \spann (\Gamma)$.
We note that if $\Gamma \subset \Z^n$ then $d(\Gamma)^2 \in \Z$.

\begin{proof}[Proof of \cref{prop:step 2}]
Let $b_{1},\ldots,b_{n-1}$ be the rows of $B$. Consider the lattice
$\Gamma$ spanned by $b_{1}\ldots,b_{n-1}$, and let $V=\spann(\Gamma)$.

Let $v\in \Z^n$ be a vector that completes $B$ to a matrix in $\SL_{n}(\Z)$. 
Then we can write $v=u+w$, with $w\in V$ and $u\in V^{\perp}$. 
We know that $\|u\|_2d(\Gamma)=1$, and since $d(\Gamma)$
is the square root of a positive integer, $\|u\|_\infty \le\|u\|_2\le1$.

We can write $w=\sum_{i=1}^{n-1}\alpha_{i}b_{i}$ for some $\alpha_i\in\R$,
and if $[\alpha_{i}]$
is the closest integer to $\alpha_{i}$ it holds that
\[
\|w-\sum[\alpha_{i}]b_{i}\|_{\infty}\le\sum\frac{1}{2}\|b_{i}\|_{\infty}\le\frac{n}{2}T.
\]
We can complete $B$ using the vector $v-\sum[\alpha_{i}]b_{i}$,
and the result follows.
\end{proof}

\subsection{Proof of \texorpdfstring{\cref{lem:skewed counting}}{Lemma 1.6}}
\label{sec:skewed counting}

We can find all elements of $\SL_n(\Z)$ with entries in the first $n-1$ rows
bounded by $T$ and entries in the last row bounded by $T^2$
by performing the following steps.

\begin{enumerate}
\item Choose a sublattice $\Gamma \subset \Z^n$ of rank $n-1$ with $d(\Gamma)\le O(T^{n-1})$, where $d(\Gamma)$ is the covolume of the lattice as above.
\item Choose a basis of $\Gamma$ of $n-1$ elements whose norm is bounded by $T$.
These will be the first $n-1$ rows of the matrix.
\item Choose a vector from a suitable translate of $\Gamma$ of norm bounded by $T^2$.
This will be the last row of the matrix.
Indeed, observe that any two possible choices for the last row differ by an element
of $\Gamma$.
\end{enumerate}

We estimate now the number of choices in each of the steps above, which will give
an upper bound for the required count.

By \cite[Theorem~2]{schmidt1968asymptotic},
\[
|\{\Gamma \subset \Z^n : \operatorname{rank} (\Gamma) = n-1, d(\Gamma) \le S\}| \le C_n S^n
\]
for all $S\in\R_{>0}$.

By \cite[Lemma~2]{schmidt1968asymptotic},
\begin{equation}\label{eq:ball count}
|\{\gamma\in \Gamma : \|\gamma\| \le T\}| \le C_n \frac{T^{n-1}}{d(\Gamma)}.
\end{equation}
Thus, we have at most $O(T^{(n-1)^2}/d(\Gamma)^{n-1})$ choices to make in the
second step for any given $\Gamma$.

Similarly, using \cref{eq:ball count} for $T^2$ in place of $T$, we
see that we have at most $O(T^{2(n-1)}/d(\Gamma))$ choices in the third step.

Recall that $d(\Gamma)$ is always the square root of an integer, so it is at least $1$.
The number of choices in the above three steps when we pick a lattice $\Gamma$ in 
the dyadic range $2^j\le d(\Gamma)<2^{j+1}$ is
\[
O\Big(2^{jn}\cdot\frac{T^{(n-1)^2}}{2^{j(n-1)}}\cdot\frac{T^{2(n-1)}}{2^j}\Big)
=O(T^{n^2-1}).
\]
We sum this up for $j=0,1,\ldots, \lceil\log_2 O(T^{n-1})\rceil$, which yields the claim.

\section{\label{sec:affine and projective}Actions on affine space and projective space}

In this section, we discuss the natural extension of the problems of this paper to the action on affine space and projective space. Namely, we consider the natural action of $\SL_n(\Z)$ on
\[
A_q = \{(x_1,\ldots,x_n)\in(\Z/q\Z)^n:\gcd(q,x_1,\ldots,x_n)=1\}
\]
(considered as column vectors) and $P_q = A_q/(\Z/q\Z)^\times$. In both cases, we give the space the distance-like function
\[
d(x,y):= \min\{\log\|\gamma\|: \gamma \in \SL_n(\Z), \gamma\cdot x=y\}.
\]
The triangle inequality is satisfied since $\|\cdot\|$ is the operator norm. Notice that the distance is not symmetric for $n\ge3$, but is symmetric for $n=2$. 

We start by considering the almost-diameter. We recall that $|A_q|= q^{n+o(1)}$, and $|P_q|=q^{n-1+o(1)}$. Given \cref{eq:DRS}, the following is optimal, and just like \cref{thm:average case}, is a special case of the optimal lifting property introduced in \cite{golubev2023sarnak}.
\begin{conj}
    For every $\epsilon>0$, and for $q$ large enough depending on $\epsilon$, the following holds.
    \begin{enumerate}
        \item For every $(x,y)\in A_q^2$ outside a set of exceptions of size $\epsilon |A_q|^2$, it holds that $d(x,y)\le (\frac{1}{n-1}+\epsilon)\log q$.
        \item For every $(x,y)\in P_q^2$ outside a set of exceptions of  size $\epsilon |P_q|^2$, it holds that $d(x,y)\le (\frac{1}{n}+\epsilon)\log q$.
    \end{enumerate}
\end{conj}

As far as we know, the only case when this conjecture is known is for $P_q$ with $q$ prime and $n=3$ \cite{kamber2023optimal}. We refer to \cite{blomer2023density,assing2022density,jana2022eisenstein-average} for some progress towards the general conjecture, which seems to be within reach, at least for $P_q$ and $q$ prime. See also \cite{assing2023density} for optimal lifting for the action on flags.

For $n=2$, stronger results are expected.
\begin{conj}\label{conj:schreier n2}
Let $n=2$. For every $\epsilon>0$, and for $q$ large enough depending on $\epsilon$, the following holds.
\begin{enumerate}
    \item For every $x\in A_q$, for all but $\epsilon|A_q|$ of $y\in A_q$, it holds that $d(x,y) \le (1+\epsilon)\log q$.
    \item For every $x\in P_q$, for all but $\epsilon|P_q|$ of $y\in P_q$ it holds that $d(x,y) \le (1/2+\epsilon)\log q$.
\end{enumerate}
\end{conj}

This conjecture follows from Selberg's eigenvalue conjecture, see \cite{golubev2019cutoff}, where this fact is stated somewhat differently. Selberg's conjecture is still open, but a recent breakthrough of Steiner \cite{steiner2023small} showed a corresponding result for (some) projective actions of quaternion algebras. Unconditionally, using the best-known bounds \cite{kim2003functoriality} on Selberg's conjecture one can deduce a theorem when the right-hand side is multiplied by $32/27$.

We now consider the diameter of the spaces. 
We introduce the limiting exponents
\begin{align*}
\overline{\alpha}(A_q) &:= \limsup_{q\to \infty} \max_{x,y\in A_q}\{d(x,y)/\log q\}, \\
\underline{\alpha}(A_q) &:= \liminf_{q\to \infty} \max_{x,y\in A_q}\{d(x,y)/\log q\}, \\
\overline{\alpha}(P_q) &:= \limsup_{q\to \infty} \max_{x,y\in P_q}\{d(x,y)/\log q\}, \\
\underline{\alpha}(P_q) &:= \liminf_{q\to \infty} \max_{x,y\in P_q}\{d(x,y)/\log q\}.
\end{align*}

In what follows, we make some observations, which allow us to determine one of the four
quantities precisely, and give some estimates for the other three.
The next statement summarizes our results.

\begin{prop}\label{prop:affine and projective}
We have
\begin{align*}
1&\le\underline{\alpha}(A_q)\le 1+\frac{1}{n},\\
1&\le\overline{\alpha}(A_q)\le 1+\frac{1}{n-1},\\
\frac{n-1}{n}&\le\underline{\alpha}(P_q)\le\overline{\alpha}(P_q)=1.
\end{align*}
\end{prop}

In the $n=2$ case, $d(\cdot,\cdot)$ is symmetric.
Moreover, \cref{conj:schreier n2}, if true, implies that any point $x\in P_q$
is of distance at most $(1/2+o(1))\log q$ from $99\%$ of the points in the space.
However, the above result gives $\overline{\alpha}(P_q)=1$, which means that there
are pairs of points of distance $(1+o(1))\log q$, at least
for some values of $q$, which is the maximum allowed by 
\cref{conj:schreier n2} up to a factor of size $1+o(1)$.

We start with the lower bounds.
\begin{proof}[Proof of the lower bounds in \cref{prop:affine and projective}]
Consider $e= (1,0,\ldots,0)\in A_q$ (recall that we always consider the vectors as column vectors, with the usual action). Then finding $\gamma \in \SL_n(\Z)$ such that $\gamma \cdot e\equiv x\equiv (a_1,a_2,\ldots,a_n)\mod q$ is the same as finding $\gamma \in \SL_n(\Z)$ with first column $(a_1,a_2,\ldots,a_n)\mod q$. In particular, there is always $x$ such that $d(e,x)\ge \log q-O(1)$, so 
$\underline{\alpha}(A_q)\ge 1$.
 
The argument for $P_q$  is similar, but we have to consider elements up to multiplication by $(\Z/q\Z)^\times$. By a pigeon-hole argument, there is always $x \in P_q$ such that $d(e,x)\ge \frac{n-1}{n} \log q-O(1)$.
This gives $\underline{\alpha}(P_q) \ge \frac{n-1}{n}$.

For $q$ of the special form $q=2q'$,
we can consider the element $x=(1,q/2,0,\ldots,0)$.
Notice that multiplication by any element of $(\Z/q\Z)^\times$
will leave the second coordinate unchanged, so $d(e,x)\ge \log q +O(1)$, and we
have
$\overline{\alpha}(P_q) \ge 1$.
\end{proof}

Our next objective is the upper bound for $P_q$.
We borrow ideas from Sarnak's proof of the $n=2$ case of \cref{thm:upper bound}
in \cite{sarnak2015lettermiller}.
We need the following lemma.
 
\begin{lemma}
    For every $x\in P_q$ there is $x'\in P_q$ with first coordinate $1$ such that 
    \[
    \max\{d(x,x'),d(x',x)\}\le 2 \log\log(q) +O(1). 
    \]
\end{lemma}
\begin{proof}
    Let $x= (a_1,\ldots,a_n)$. We will choose $\gamma \in \SL_n(\Z)$ with first row $(1,\alpha_2,\ldots,\alpha_n)$ and the other rows agree with the identity matrix, so $x' = \gamma x = (a_1+\sum \alpha_i a_i,a_2,\ldots,a_n)$.
    
    The goal is to choose $\alpha_2,\ldots,\alpha_n$ so that $a_1' = a_1+\sum \alpha_i a_i$ will be coprime to $q$, since then $x'$ is equivalent to an element with first coordinate $1$. Denote $d_2 = \gcd(q,a_1,a_2)$ and for $i=3,\ldots,n$, $d_i = \gcd(d_{i-1},a_i)$. Iwaniec's bound \cite{iwaniec1978} on the Jacobsthal function implies that the arithmetic progression $a_1+\alpha a_2$, $\alpha=0,1,\ldots$ contains an element with $\gcd(q,a_1+\alpha a_2)=d_2$ and $\alpha=O((\log q)^2)$, see \cite[Page 8]{sarnak2015lettermiller} for more details. We pick $\alpha_2$ to satisfy these properties. Then iteratively for $j=3,\ldots,n$, we find $\alpha_j = O((\log q)^2)$ such that $\gcd(q,a_1 +\sum_{i=1}^j \alpha_i a_i)=d_j$. For $j=n$, this implies that $\gcd(q,a_1+\sum \alpha_i a_i )=d_n=1$. Then
    \[
    \max\{\log \|\gamma\|, \log \|\gamma^{-1}\|\} \le 2 \log\log(q)+O(1),
    \]
    and the claim follows.
\end{proof}

\begin{proof}[Proof of $\overline{\alpha}(P_q)\le 1$]
Let $x,y\in P_q$.
We use the lemma to find $x',y'\in P_q$ with first coordinates equal to $1$ and
\[
\max\{d(x,x'),d(y',y)\}\le 2 \log\log(q) +O(1).
\]
Notice that there is a unipotent matrix $\gamma \in \SL_n(\Z)$ with coordinates bounded by $q$ such that $\gamma \cdot x'=y'$. Then 
\[
d(x,y) \le d(x,x')+d(x',y')+d(y',y) \le \log q + 4 \log \log q +O(1),
\]
which proves $\overline{\alpha}(P_q)\le 1$.
Recall that the triangle inequality is always satisfied by $\dist(\cdot,\cdot)$
on $P_q$ even when it is not a metric.
\end{proof}

It remains to prove the upper bounds on $\underline{\alpha}(A_q)$ and
$\overline{\alpha}(A_q)$.
These follow immediately from the following two lemmata.

We recall that $d(\Gamma)$ denotes the covolume of a lattice $\Gamma \subset \R^n$ of rank $k$. By Minkowski's theorem (\cite[Lecture II, \S 1]{siegel1989lectures}), there is a constant $C_k$ such that there is some $v\in\Gamma\backslash\{0\}$ with
\begin{equation}\label{eq:short element}
	\|v\|\le (C_k d(\Gamma))^{1/k}.
\end{equation}

Given $\Gamma\subset\Z^n$, we can define the perpendicular lattice by
\[
\Gamma^\perp = \{\gamma'\in \Z^n :  \langle \gamma, \gamma'\rangle = 0\text{ for all $\gamma \in \Gamma$}\}.
\]
By \cite[Corollary of Lemma~1]{schmidt1968asymptotic}, 
\begin{equation}\label{eq:perp lattice}
	d(\Gamma)=d(\Gamma^\perp)
\end{equation}
if $\Gamma$ is primitive.

\begin{lemma}
    For every $x\in A_q$, there is $x'\in A_q$ with first coordinate $0$ such that $d(x,x')\le (1/(n-1))\log q+O(1)$.
    
    Moreover, if $q$ is prime, then we can improve the bound to  $(1/n)\log q+O(1)$.
\end{lemma}
\begin{proof}
    Let $x =(a_1,\ldots,a_n)$. Define the lattice
    \[
    \Gamma = \{(b_1,\ldots,b_n)\in \Z^n : \sum \wt a_i b_i =0\}.
    \]
    This is a primitive lattice of rank $n-1$, and by \cref{eq:perp lattice},
    \[
    d(\Gamma)\le \|(\wt a_1,\ldots,\wt a_n)\| = O(q).
    \]
    (Note that $(\wt a_1,\ldots, \wt a_n)$ may be imprimitive.)
    By \cref{eq:short element}, there is $(b_1,\ldots,b_n)\in \Gamma$ with $|b_i| =O(q^{1/(n-1)})$. There exists a matrix $\gamma \in \SL_n(\Z)$ with first row $(b_1,\ldots,b_n)$, and one can actually find such a matrix with norm bounded by $O(\max_i \{|b_i|\})= O(q^{1/(n-1)})$. We leave this fact to the reader. (One way to prove this is to extend the vector $(b_1,\ldots,b_n)$ to a vector space basis using some short vectors in $\Z^d$ and then apply the procedure in \cite[Lecture V \S 2]{siegel1989lectures}. See also \cite[Lecture X \S 4]{siegel1989lectures}.) This implies that $x' = \gamma \cdot x$ is the required element, with $d(x,x') \le \frac{1}{n-1}\log q +O(1)$. 

    If $q$ is prime, we can look at the lattice
    \[
    \Gamma = \{(b_1,\ldots,b_n)\in \Z^n : \sum \wt a_i b_i =0 \mod q\}. 
    \]
    This is a lattice of full rank and index $q$ in  $\Z^n$, so $d(\Gamma)=q$, and by \cref{eq:short element}, there is $(b_1,\ldots,b_n)\in \Gamma$ with $|b_i| =O(q^{1/n})$. Denote $d=\gcd(b_1,..,b_n)$. If $q$ is large enough, $d<q$, and since $q$ is prime, $d\in (\Z/q\Z)^\times$. This implies that $\frac{1}{d}(b_1,\ldots,b_n) \in \Gamma$. Since the $\gcd$ of all elements of the last vector is $1$, we can continue as before, finding $\gamma \in \SL_n(\Z)$ with norm bounded by $O(q^{\frac{1}{n}})$ and first row  $\frac{1}{d}(b_1,\ldots,b_n)$. Then $x'=\gamma \cdot x$ has first coordinate $0$ and $d(x,x')\le \frac{1}{n}\log q + O(1)$.
\end{proof}

\begin{lemma}
If $x'\in A_q$ has first coordinate 0 and $y \in A_q$ is arbitrary, then $d(x',y)\le \log q+\log \log q+O(1)$.
\end{lemma}
\begin{proof}
    There is a matrix $A \in \SL_n(\Z/q\Z)$ such that $Ax' = y$. By \cref{prop:step 1,,prop:step 2}, there is a matrix $\gamma \in \SL_n(\Z)$ with norm bounded by $O(q\log q)$ such that $\gamma \mod q$ agrees with $A$ on the last $n-1$ columns. Since the first coordinate of $x'$ is $0$, this implies that $\gamma x'=y$. The claim follows.
\end{proof}

\bibliographystyle{acm}
\bibliography{./database}

\end{document}